\documentclass{amsart}
\usepackage{verbatim,amsmath}

\newtheorem{theorem}{Theorem}
\newtheorem{proposition}{Proposition}
\newtheorem{corollary}[proposition]{Corollary}
\newtheorem{lemma}[proposition]{Lemma}

\theoremstyle{definition}
\newtheorem{definition}[proposition]{Definition}
\newtheorem{example}[proposition]{Example}
\newtheorem*{remark}{Remark}

\newcommand{\ZZ}{{\mathbb Z}}
\newcommand{\RR}{{\mathbb R}}
\newcommand{\QQ}{{\mathbb Q}}

\newcommand{\tog}{{S}}
\newcommand{\rev}{{R}}

\newcommand{\wtc}{{\mathsf{wtc}}}

\newcommand{\qqand}{{\qquad\text{and}\qquad}}
\newcommand{\qqiff}{{\qquad\text{if and only if}\qquad}}

\renewcommand{\AA}{{\mathcal A}}
\newcommand{\vect}[1]{{\mathbf{#1}}}

\begin{document}

\title{Orders on free groups induced by oriented words}
\author{Zoran \v{S}uni\'c}
\address{Dept. Of Mathematics, Texas A\&M University, College Station, TX 77843-3368, USA}
\email{sunic@math.tamu.edu}
\thanks{This material is based upon work supported by the National Science
Foundation under Grant No. DMS-1105520. }

\begin{abstract}
For every finite rank $k$, $k \geq 2$, we explicitly construct $(2k)!$ left orders on the free group $F_k$ of rank $k$. Each order is induced by a word of length $2k$ in which each generator of $F_k$ and its inverse appear exactly once. For each of these $(2k)!$ words we define a real valued function on $F_k$, which is shown to be a quasi-character with small relative defect and which is used as a weight function to define the corresponding order (the elements of $F_k$ which evaluate to positive real numbers are declared positive in the group). Some of the orders we define on $F_k$ are extensions of the usual lexicographic order on the positive monoid and some have word reversible positive cones. We characterize the defining words leading to orders of either of these two types.
\end{abstract}

\dedicatory{To Smile Markovski, my teacher and friend, wishing him long and enjoyable retirement}


\maketitle

\section{Introduction}

For each finite rank $k$, $k \geq 2$, one particularly easy to describe left order on the free group $F_k$ was introduced in~\cite{sunic:free-lex}. When $k=2$ and $F_2=F(a,b)$, the order is defined as follows. The element represented by a reduced group word $g$ over $\{a,b\}$ is positive if the number of occurrences of $ba^{-1}$ in $g$ is strictly greater than the number of occurrences of $b^{-1}a$ in $g$, and it is negative if the number of occurrences if $ba^{-1}$ in $g$ is strictly smaller than the number of occurrences of $b^{-1}a$ in $g$. In case these two numbers are equal, the last letter of $g$ is the tie-breaker, namely, $g$ is positive if the last letter in $g$ is $a$ or $b$ and negative if the last letter is $a^{-1}$ or $b^{-1}$.

It was announced in~\cite{sunic:free-lex} that variations of the order described above can easily be introduced and one of the purposes of this work is to explicitly present those variations. Namely, we describe, for every $k \geq 2$, $(2k)!$ different left orders on $F_k$, induced by the $(2k)!$ oriented words of rank $k$ (these are the words of length $2k$ in which each generator of $F_k$ and its inverse appear exactly once; Section~\ref{s:background}, Definition~\ref{d:oriented}).

A way of constructing actions of the free group, free on at least one orbit, was provided in~\cite{sunic:free-pda} and, as an application, examples of actions of $F_k$, $k \geq 2$, on the circle, free on the orbit of 0, were also given there~\cite[Subsection 2.1]{sunic:free-pda}, along with a passing remark that these actions could be used to define orders on $F_k$ (by lifting them to actions on the line). The orders introduced here are precisely the orders that correspond to those actions. However, the approach we take here avoids the discussion of the actions altogether and defines, for each oriented word $u$, the associated order $\leq_u$ directly in terms of a certain weight function $\tau_u:F_k \to \RR$ (Section~\ref{s:orders}, Definition~\ref{d:weight}), which is, in turn, described solely in terms of the structure of the defining oriented word $u$. Positive elements under the order $\leq_u$ are precisely those whose weight under $\tau_u$ is positive (Section~\ref{s:orders}, Definition~\ref{d:order}). One of the reasons we decided to suppress any direct reference to the corresponding actions on the line is to provide a different perspective and concentrate on the weight functions, which are not only sufficient to define the orders in simple terms, but provide an excellent tool to study them. The point of view that emphasizes the actions rather than the weights is taken in a different work~\cite{sunic:from-bi-infinite}, where Cantor sets of orders on free groups are constructed, each induced by a bi-infinite word (in a sense, the orders defined here are ``strung together'' in various ways to build uncountably many orders).

A non-obvious but useful feature of the weight functions is that they are quasi-characters of the free group (Section~\ref{s:quasi}, Theorem~\ref{t:tau-is-quasi}). Moreover, the defect of the weight functions is relatively small (when compared to the values of the weight functions), and this property is used to show that they indeed define orders on free groups (Section~\ref{s:criterion}, Theorem~\ref{t:is-order}). The proof is based on a simple general condition under which a function may be used as a weight function defining an order on a group (Proposition~\ref{p:criterion}).

We characterize the oriented words for which the induced order on $F_k$ is the extension of the lexicographic order on the free monoid or rank $k$ (Section~\ref{s:lex}, Proposition~\ref{p:lexicographic}), and we also characterize the oriented words that induce reversible orders on $F_k$ (Section~\ref{s:rev}, Proposition~\ref{p:reversible}). Moreover, we establish an interesting relation between the orders $\leq_u$ and $\leq_{u^{-1}}$, namely, the positive elements under $\leq_{u^{-1}}$ are precisely the word reversals of the positive elements under $\leq_u$ (Section~\ref{s:rev}, Theorem~\ref{t:inverse-reverse}).

At the end, we show that, for $k \geq 2$, none of the $(2k)!$ orders defined on $F_k$ by the $(2k)!$ oriented words is a bi-order (Section~\ref{s:no-bi}, Proposition~\ref{p:no-bi}), and that they are all distinct (Section~\ref{s:different}, Theorem~\ref{t:different}).

\subsection*{Acknowledgment}

I would like to thank Mark Sapir and Slava Grigorchuk for sharing their thoughts on quasi-characters with me, and Warren Dicks for his continuing interest and encouragement.

%---------------------------------------------------------------------------------

\section{Background, general setting, and notation}\label{s:background}

\subsection{Ordered groups}

A \emph{left order} on a group $G$ is a linear order on $G$ that is compatible with multiplication on the left in $G$, i.e., for all $f,g,h \in G$,
\[
 g \leq h \implies fg \leq fh.
\]
A \emph{right order} is a linear order that is compatible with multiplication on the right, and a \emph{bi-order} is a linear order that is simultaneously left and right compatible. We are interested in left orders in this work and whenever we use the word order we mean a left order. A group is \emph{orderable} if it admits an order.

It is common to study and discuss, and often even define, orders through their corresponding positive cones. Given an order $\leq$ on $G$, the \emph{positive cone} of the order is the set
\[
 P = \{ \ g \mid e < g \ \}
\]
of elements that are strictly greater than the identity. The positive cone satisfies the following two properties:
\begin{itemize}
 \item[](closure) if $g,h \in P$ then $gh \in P$,
 \item[](trichotomy) $G$ decomposes into the disjoint union $G= P^{-1} \sqcup \{e\} \sqcup P$.
\end{itemize}
Conversely, any subset $P \subseteq G$ that satisfies the closure and trichotomy properties uniquely determines an order on $G$ for which $P$ is the positive cone, namely,
\[
 g < h \iff g^{-1}h \in P.
\]

It is well known that the free group of rank $k$, $k \geq 2$, is bi-orderable~\cite{shimbireva:ordered,iwasawa:ordered,neumann:ordered,vinogradov:ordered}. The bi-orders introduced in~\cite{shimbireva:ordered,iwasawa:ordered,neumann:ordered} are based on the lower central series of $F_k$ (each quotient in the series is torsion-free abelian, hence bi-orderable) and the ones in~\cite{vinogradov:ordered} on embeddings of $F_2$ into a certain ring of infinite matrices with coefficients in the group ring of $\ZZ^2$. It was realized later that the Magnus embedding~\cite{magnus:free-embedding} of $F_k$ in the (bi-ordered) ring $\ZZ\ll x_1,x_2,\dots,x_k\gg$ of power series with integer coefficients over $k$ non-commuting variables induces an order on $F_k$, and this is probably the most widely known order on $F_k$.

The orders induced by oriented words that we define here are considerably easier to describe in concrete terms than any of the bi-orders mentioned in the last paragraph. Indeed, as we will see, deciding if an element is positive amounts to a simple signed count of subwords of length 2. On the other hand, none of the orders induced by oriented words is a bi-order. It should be noted, however, that some of the interesting properties enjoyed by some of the orders induced by oriented words cannot be shared by bi-orders. For instance, no bi-order on $F_k$, $k \geq 2$, can be an extension of the lexicographic order on the free monoid, nor can a bi-order on $F_k$, $k \geq 2$, be reversible.

\subsection{General setting/notation}
The free group of rank $k$ is denoted by $F_k$. Its basis, unless specified otherwise, is $\AA_k=\{a_1,\dots,a_k\}$. The letters from the beginning of alphabet, such as $a,b,c,\dots$, usually denote elements from the basis $\AA_k$. To simplify notation, when $a$ ($b,c,\dots$) denotes an element in $\AA$, its inverse is denoted by the corresponding uppercase letter $A$ ($B,C,\dots$). For this reason, the letters in $\AA_k$ are sometimes called lowercase letters, and those in $\AA^{-1}_k$ uppercase letters. A group word over $\AA_k$ is a word over $\AA_k^\pm = \{a_1,\dots,a_k,A_1,\dots,A_k\}$. Such words are called group words of rank $k$. We sometimes omit the index $k$ in $\AA_k$.

By choice, the letters in $\AA_k$ are positive elements with respect to all orders defined here. Accordingly, the letters in $\AA_k^{-1}$ are always negative. This is compatible with the common way of referring to the elements in the basis $\AA_k$ as positive generators (letters). Technically speaking these two notions of positivity are independent, but we find it convenient to keep them in sync in this work (renaming the symbols is always an option). While many of our claims and results make sense and are correct when $k=1$, some are not, and we will avoid discussing this trivial case. Thus, $k$ is considered to be greater than 1, even if we do not mention this explicitly.

\begin{definition}\label{d:oriented}
An \emph{oriented word} over $\AA_k$ (an oriented word of rank $k$), $k \geq$, is a group word of length $2k$ that contains each of the letters in $\AA_k^\pm$ exactly once.
\end{definition}

For any two distinct letters $x$ and $y$ in $\AA_k^\pm$ and any oriented word $u$ of rank $k$, we write $x..y$ in $u$, if $x$ is to the left of $y$ in $u$.

%---------------------------------------------------------------------------------

\section{Weights and orders induced by oriented words}\label{s:orders}

For every oriented word of rank $k$, $k \geq 2$, we define an order $\leq_u$ on $F_k$. The order $\leq_u$ is defined in terms of a weight function $\tau_u: F_k \to \RR$. The weight function is based on the structure of the defining word $u$ and it has three components as defined below. Note that the third of these components, the last letter weight, does not depend on $u$.

\begin{definition}[Weight induced by $u$]\label{d:weight}
Let $u$ be an oriented word of rank $k$, $k \geq 2$.

Define the \emph{case transition weight} $\alpha_u: F_k \to \RR$ as follows. For a reduced group word $g$ of rank $k$,
\begin{equation}\label{e:alpha}
 \alpha_u(g) = \sum_{\substack{a,b \in \AA \\ A..B \textup{ in } u}} \#_{aB}(g) - \sum_{\substack{a,b \in \AA \\ a..b \textup{ in } u}} \#_{Ba}(g).
\end{equation}

Define the \emph{letter transition weight} $\beta_u: F_k \to \RR$ as follows. For a reduced group word $g$ of rank $k$,
\begin{equation}\label{e:beta}
 \beta_u(g) = \sum_{\substack{a,b \in \AA \\ A..b \textup{ in } u}} \#_{ab}(g) - \sum_{\substack{a,b \in \AA \\ A..b \textup{ in } u}} \#_{BA}(g).
\end{equation}

Define the \emph{last letter weight} $\omega: F_k \to \RR$ as follows. For a reduced group word $g$ of rank $k$,
\begin{equation}\label{e:omega}
 \omega(g) = \frac{1}{2}
  \begin{cases}
   1, & \text{if the last letter of } g \text{ is positive (i.e., it is in } \AA) \\
  -1, & \text{if the last letter of } g \text{ is negative (i.e., it is in } \AA^{-1})\\
   0, & \text{if } g \text{ is trivial}
  \end{cases}.
\end{equation}

Finally, define the \emph{weight} function (induced by the oriented word $u$) $\tau_u: F_k \to \RR$, by
\[
 \tau_u = \alpha_u + \beta_u + \omega.
\]
\end{definition}

\begin{remark}
We emphasize that, as written, the formulae~\eqref{e:alpha},~\eqref{e:beta} and~\eqref{e:omega} defining the case transition, the letter transition, and the last letter weights are not invariant under free reductions/insertions and they should only be applied to words in reduced form. For an arbitrary group word $g$ of rank $k$, $\tau_u(g)$ is equal to $\tau_u(g')$ where $g'$ is the reduced form of $g$ and~\eqref{e:alpha},~\eqref{e:beta} and~\eqref{e:omega} can and should be applied to $g'$.
\end{remark}

\begin{definition}[Positive cone and order induced by $u$]\label{d:order}
Given an oriented word $u$ of rank $k$, $k \geq 2$, let
\[
 P_u = \{\ g \in F_k \mid \tau_u(g)>0 \ \}.
\]
A left order $\leq_u$ on $F_k$ induced by the oriented word $u$ is defined by declaring $P_u$ to be the positive cone of the order.
\end{definition}

\begin{remark}
It is not immediately clear that $\leq_u$ is actually an order on $F_k$ and this will be proved in Section~\ref{s:criterion}, Theorem~\ref{t:is-order}.
\end{remark}

\begin{example}\label{e:abBA}
We recast the example from the introduction in the setting we just introduced. 

Let $F_2 = F(a,b)$ and $u=abBA$. The weight function on $F_2$ induced by $u$ is given by
\[
 \tau_u(g) = \#_{bA}(g) - \#_{Ba}(g) + \omega(g),
\]
for $g$ a reduced group word over $\{a,b\}$. For instance, $\tau_u(BAbA) = \#_{bA}(BAbA) - \#_{Ba}(BAbA) + \omega(BAbA) = 1-0 -1/2 = 1/2$.

The positive cone induced by $u$ is
\[
 P_u = \{ \ u \in F_2 \mid \tau_u(g)>0 \ \}.
\]
The fact that $\tau_u(BAbA) = 1/2>0$ indicates that $BAbA$ is a positive element and, in particular, $ab <_u bA$. From smallest to largest, the elements of length at most 2 are
\[
 aB,~BB,~BA,~B,~Ba,~AB,~AA,~A,~e,~a,~aa,~ab,~bA,~b,~ba,~bb,~Ab.
\]
\end{example}

We consider another way to look at the weight function, through weight contributions.

\begin{definition}[Weight contributions induced by $u$]
Let $u$ be an oriented word of rank $k$, $k \geq 2$. For every reduced 2-letter word $xy$ of rank $k$ define the \emph{weight contribution} $\wtc_u(xy)$ of $xy$ induced by $u$, as follows. If $x=a$ is a positive letter, then, for every positive or negative letter $y \in \AA_k^\pm$, $y \neq A$,
\[
 \wtc_u(ay) =
  \begin{cases}
    1, & \textup{if } A..y \textup{ in } u \\
    0, & \textup{if } y..A \textup{ in } u
  \end{cases}
\]
and, if $x=A$ is a negative letter, then, for every positive or negative letter $y \in \AA_k^\pm$, $y \neq a$,
\[
 \wtc_u(Ay) =
  \begin{cases}
    -1, & \textup{if } y..a \textup{ in } u \\
    0, &  \textup{if } a..y \textup{ in } u
  \end{cases}.
\]
\end{definition}

\begin{remark}
It is important to observe that the weight contribution of a reduced 2-letter word that starts in a positive letter is always nonnegative, 0 or 1, and the weight contribution of a reduced 2-letter word that starts in a negative letter is always nonpositive, 0 or -1. We use Table~\ref{tbl:wtc} as a lookup whenever we need to make quick calculations or comparisons.
\begin{table}[!ht]
\begin{alignat*}{7}
 &A..B \qquad &&\iff \qquad &&\wtc(aB)&&=1 \qquad &&\iff \qquad &&\wtc(bA)&&=0 \\
 &a..b \qquad &&\iff \qquad &&\wtc(Ab)&&=0 \qquad &&\iff \qquad &&\wtc(Ba)&&=-1 \\
 &A..b &&\iff &&\wtc(ab)&&=1 &&\iff &&\wtc(BA)&&=-1  \\
 &a..B &&\iff &&\wtc(AB)&&=0 &&\iff &&\wtc(ba)&&=0
\end{alignat*}
\caption{Weight contributions associated to relative positions of letters}
\label{tbl:wtc}
\end{table}
The top row in Table~\ref{tbl:wtc} indicates that, for any two (necessarily distinct) positive letters $a,b\in \AA$ the weight contribution $\wtc_u(bA)$ is 0 if and only if $A$ is to the left of $B$ in $u$. Similarly, the third row indicates that, for any two (not necessarily distinct) positive letters $a,b \in \AA$, $A$ is to the left of $b$ in $u$ if and only of the weight contribution of $BA$ is -1.
\end{remark}

\begin{proposition}
Let $u$ be an oriented word of rank $k$, $k \geq 2$. For a reduced group word $g=x_1x_2\dots x_n$ of rank $k$,
\[
 \tau_u(g) = \sum_{i=1}^{n-1} \wtc_u(x_ix_{i+1})  + \omega(g).
\]
\end{proposition}

\begin{proof}
The top two rows in Table~\ref{tbl:wtc} correspond to the case transition part $\alpha_u$, and the bottom two to the letter transition part $\beta_u$ of the weight $\tau_u$ induced by $u$.
\end{proof}

%--------------------------------------------------------------------------

\section{The weight functions as quasi-characters}\label{s:quasi}

Recall that a (real) \emph{quasi-character} (terminology from~\cite{grigorchuk:bounded}) of a group $G$ is a function $\varphi: G \to \RR$ such that there exists a constant $D \in \RR$ such that, for all $g,h \in G$,
\[
 | \varphi(g)+\varphi(h)- \varphi(gh) | \leq D.
\]
The value $\delta(\varphi)=\sup_{g,h \in G} | \varphi(g)+\varphi(h)- \varphi(gh) |$ is called the \emph{defect} of $\varphi$. It seemingly measures how far from a (real) character $\varphi$ is, but this is not correct, as there are groups that admit quasi-characters that do not differ by a bounded function from any character (Johnson~\cite{johnson:cohomology-banach} observed this for the free group). In fact, from the bounded cohomology point of view, quasi-characters that differ from a character by a bounded function are considered trivial.

The purpose of this section is to prove the following result.

\begin{theorem}\label{t:tau-is-quasi}
For every oriented word $u$ of rank $k$, $k \geq2$, the weight $\tau_u : F_k \to \RR$ is a quasi-character of defect $1/2$.

Moreover,
\[
 | \tau_u(g)+\tau_u(h)- \tau_u(gh) | =
   \begin{cases}
    0,           &\textup{if } g=e     \textup{ or }  h=e      \textup{ or }  gh=e, \\
    \frac{1}{2}, &\textup{if } g\neq e,~h \neq e \textup{ and } gh \neq e.
   \end{cases}
\]
\end{theorem}

\begin{remark}
Recall that, for every reduced group word $w$ of rank $k$, Brooks~\cite{brooks:bounded} defined a quasi-character $\phi_w:F_k\to\RR$ by
\[
 \phi_w(g) = \#_{w}(g) - \#_{w^{-1}}(g),
\]
for $g$ a reduced group word of rank $k$. The letter transition weight is a finite sum of Brooks quasi-characters
\[
 \beta_u = \sum_{\substack{a,b \in \AA \\ A..b \textup{ in } u}} \phi_{ab}.
\]
The letter transition weight $\beta_u$ vanishes precisely when all positive letters precede all negative letters in the oriented word $u$ (as in Example~\ref{e:abBA}). The number of Brooks quasi-characters comprising $\beta_u$ varies between 0 and $k^2$, depending on the relative placement of the lowercase and the uppercase letters in $u$.

On the other hand, the case transition weight function $\alpha_u$ does not vanish, regardless of $u$. The expression for the case transition function $\alpha_u$ always has exactly $\binom{k}{2}$ plus terms, related to the relative placement of the uppercase letters, and $\binom{k}{2}$ minus terms, related to the relative placement of the lowercase letters in $u$.

For instance, when $k=2$, the case transition weight $\alpha_u$ always has exactly one plus term, exactly one minus term, and the number of Brooks summands in the letter transition weight $\beta_u$ is anywhere between 0 and 4. There are exactly 4 oriented words for which the letter transition function vanishes and they are the word $abBA$ from Example~\ref{e:abBA}, the word $abAB$, for which the weight is given by
\[
 \tau_{abAB}(g) = \#_{aB}(g) - \#_{Ba}(g) + \omega(g),
\]
for $g$ a reduced word of rank $2$, and the two words $baAB$ and $baBA$ obtained from the previous two by reversing the roles of $a$ and $b$.
\end{remark}

The rest of the section is devoted to a proof of Theorem~\ref{t:tau-is-quasi}.

A nontrivial reduced word of rank $k$ is of the form $\underline{\hspace{3mm}}m$ if it ends in a lowercase letter, it is of the form $m\underline{\hspace{3mm}}M$ if it starts in a lowercase letter and ends in an uppercase letter, and so on.

\begin{lemma}\label{l:total-transition-count}
Let $g$ be a nontrivial reduced group word of rank $k$, $k \geq2$. Then
\[
 \sum_{\substack{a,b \in \AA \\ a \neq b}} \#_{aB}(g) - \sum_{\substack{a,b \in \AA \\ a \neq b}} \#_{Ab}(g) =
 \left\{
  \begin{array}{rr}
   0, & \textup{if } g \textup{ is of the form } m\underline{\hspace{3mm}}m \textup{ or } M\underline{\hspace{3mm}}M, \\
   1, & \textup{if } g \textup{ is of the form } m\underline{\hspace{3mm}}M, \\
  -1, & \textup{if } g \textup{ is of the form } M\underline{\hspace{3mm}}m.
  \end{array}
 \right.
\]
\end{lemma}

\begin{proof}
The quantity calculated by the expression on the left hand side is the number of case transitions from a lowercase letter to an uppercase letter minus the number of case transitions
from an uppercase letter to a lowercase letter. Since such transitions must alternate, this difference can only be 0,1, or -1, as indicated.
\end{proof}

\begin{lemma}\label{l:gginv}
Let $u$ be an oriented word of rank $k$, $k \geq 2$. For every reduced word $g$ of rank $k$,

(a)
\[
 \alpha_u(g) + \omega(g) + \alpha_u(g^{-1}) + \omega(g^{-1}) = 0.
\]

(b)
\[
 \beta_u(g) + \beta_u(g^{-1}) = 0.
\]

(c)
\[
 \tau_u(g) + \tau_u(g^{-1}) = 0.
\]
\end{lemma}

\begin{proof}
(a) The quantity $\alpha_u(g^{-1}) + \omega(g^{-1})$ is equal to
\[
 \sum_{\substack{a,b \in \AA \\ A..B \textup{ in }w}} \#_{aB}(g^{-1}) - \sum_{\substack{a,b \in \AA \\ a..b \textup{ in }w}} \#_{Ba}(g^{-1})
                 + \left\{
                 \begin{array}{rr}
                 1/2, & \textup{if } g^{-1} \textup{ is of the form } \underline{\hspace{3mm}}m, \\
                 -1/2, & \textup{if } g^{-1} \textup{ is of the form } \underline{\hspace{3mm}}M,
                 \end{array}
                 \right.
\]
which is in turn equal to
\[
    \sum_{\substack{a,b \in \AA \\ A..B \textup{ in }w}} \#_{bA}(g) - \sum_{\substack{a,b \in \AA \\ a..b \textup{ in }w}} \#_{Ab}(g)
                 + \left\{
                 \begin{array}{rr}
                 1/2, & \textup{if } g \textup{ is of the form } M\underline{\hspace{3mm}}, \\
                 -1/2, & \textup{if } g \textup{ is of the form } m\underline{\hspace{3mm}}.
                 \end{array}
                 \right.,
\]
and, after reindexing (exchanging the roles of $a$ and $b$), to
\[
    \sum_{\substack{a,b \in \AA \\ B..A \textup{ in }w}} \#_{aB}(g) - \sum_{\substack{a,b \in \AA \\ b..a \textup{ in }w}} \#_{Ba}(g)
                 + \left\{
                 \begin{array}{rr}
                 1/2, & \textup{if } g \textup{ is of the form } M\underline{\hspace{3mm}}, \\
                 -1/2, & \textup{if } g \textup{ is of the form } m\underline{\hspace{3mm}}.
                 \end{array}
                 \right..
\]
Adding the expression we just obtained to $\alpha_u(g) + \omega(g)$ as given in Definition~\ref{d:weight}, and using Lemma~\ref{l:total-transition-count}, yields the result.

(b) Recall that $\beta_u$ is a sum of Brooks quasi-characters, and $\phi_w(g) + \phi_w(g^{-1})=0$, for any Brooks quasi-character $\phi_w$.

(c) Follows directly from (a) and (b).
\end{proof}

\begin{remark}
Note that Lemma~\ref{l:gginv}(c) says that, for every oriented word $u$, the weight $\tau_u$ is symmetric, i.e., weights assigned to $g$ and $g^{-1}$ are equal in absolute value and opposite in sign. We sometimes use this fact in the rest of the text without reference.
\end{remark}

For a nontrivial reduced word $f$ or rank $k$ and a (positive or negative) letter $x \in \AA^\pm$, we write $f=*x$ to indicate that $x$ is the last letter of $f$.

\begin{lemma}\label{l:defect}
Let $u$ be an oriented word of rank $k$, $k \geq 2$. 

(a) If $x$ and $y$ are two distinct letters in $\AA^\pm$, and $f=*x$ and $g=*y$ are two nontrivial reduced words of rank $k$, then
\[
 \tau_u(fg^{-1}) - \tau_u(f) - \tau_u(g^{-1}) = \wtc_u(xy^{-1}) - \omega(x) =
  \left\{
  \begin{array}{rr}
    1/2,  & x^{-1}..y^{-1} \textup{ in } u, \\
   -1/2,  & y^{-1}..x^{-1} \textup{ in } u.
  \end{array}
  \right.
\]

(b) If $fg^{-1}$, $gh^{-1}$, and $hf^{-1}$ are three nontrivial reduced words of rank $k$, then
\[
 |\tau_u(fg^{-1}) + \tau_u(gh^{-1}) + \tau_u(hf^{-1})| = \frac{1}{2}.
\]
\end{lemma}

\begin{proof}
(a) Note that $g^{-1} = y^{-1}*$. Therefore,
\[
 \tau_u(fg^{-1}) = \tau_u(f) - \omega(x) + \wtc_u(xy^{-1}) + \tau_u(g^{-1}),
\]
which implies that $\tau_u(fg^{-1}) - \tau_u(f) - \tau_u(g^{-1}) = \wtc_u(xy^{-1}) - \omega(x)$. If $x^{-1}..y^{-1}$ in $u$ and $x$ is a lowercase letter, then $\wtc_u(xy^{-1})=1$, and $\wtc_u(xy^{-1}) - \omega(x)=1/2$, and if $x$ is an uppercase letter, then $\wtc_u(xy^{-1})=0$, and $\wtc_u(xy^{-1}) - \omega(x)=1/2$. If $y^{-1}..x^{-1}$ in $u$ and $x$ is a lowercase letter, then $\wtc_u(xy^{-1})=0$, and $\wtc_u(xy^{-1}) - \omega(x)=-1/2$, and if $x$ is an uppercase letter, then $\wtc_u(xy^{-1})=-1$, and $\wtc_u(xy^{-1}) - \omega(x)=-1/2$.

(b) If any of the words $f$, $g$, or $h$ is trivial, say $h$, then the other two must be nontrivial and
\[
 |\tau_u(fg^{-1}) + \tau_u(gh^{-1}) + \tau_u(hf^{-1})| = |\tau_u(fg^{-1}) - \tau_u(g^{-1}) - \tau_u(f)| = \frac{1}{2},
\]
where the last equality follows from part (a). If none of the words $f$, $g$, and $h$ is trivial, then there exist three distinct letters $x,y,z\in\AA^\pm$ such that $f=*x$, $g=*y$, and $h=*z$. Since $\tau_u(fg^{-1}) + \tau_u(gh^{-1}) + \tau_u(hf^{-1})$ is equal to
\[
 [\tau_u(fg^{-1}) - \tau_u(f) - \tau_u(g^{-1})] + [\tau_u(gh^{-1}) - \tau_u(g) - \tau_u(h^{-1})] +[\tau_u(hf^{-1}) - \tau_u(h) - \tau_u(f^{-1})],
\]
and part (a) applies to each of the three summands (delimited by square brackets), it follows that $|\tau_u(fg^{-1}) + \tau_u(gh^{-1}) + \tau_u(hf^{-1})| \neq 1/2$ only if the three summands have the same sign. However, this would imply $x^{-1}..y^{-1}..z^{-1}..x^{-1}$ in $u$ (if all three summands are positive) or $x^{-1}..z^{-1}..y^{-1}..x^{-1}$ in $u$ (if all three summands are negative), neither of which is possible.
\end{proof}

\begin{proof}[Proof of Theorem~\ref{t:tau-is-quasi}]
If $g=e$ or $h=e$, then $| \tau_u(g)+\tau_u(h)- \tau_u(gh) | = \tau_u(e)=0$. If $gh=e$, then, by Lemma~\ref{l:gginv}(c), $|\tau_u(g)+\tau_u(h)- \tau_u(gh) | = 0$. The claim that $| \tau_u(g)+\tau_u(h)- \tau_u(gh) |$ is equal to $1/2$, when $g \neq e$, $h \neq e$, and $gh \neq e$, is Lemma~\ref{l:defect}(b) in disguise. Indeed, for any two nontrivial reduced words $g$ and $h$ for which $gh \neq e$, there exists three reduced words $f_1$, $g_1$, and $h_1$ such that the reduced forms of $g$, $h$ and $gh$ are $g=g_1h_1^{-1}$, $h=h_1f_1^{-1}$, and $gh=g_1f_1^{-1}$, respectively (the subwords $h_1^{-1}$ in $g$ and $h_1$ in $h$ represent the parts of $g$ and $h$ that are canceled in the product $gh$ to reach its reduced form $g_1f_1^{-1}$). By Lemma~\ref{l:gginv}(c) and Lemma~\ref{l:defect}(b),
\[
 | \tau_u(g)+\tau_u(h)- \tau_u(gh) | = | \tau_u(g_1h_1^{-1}) + \tau_u(h_1f_1^{-1}) + \tau_u(f_1g_1^{-1}) | = 1/2. \qedhere
\]
\end{proof}

%--------------------------------------------------------------------------

\section{Functions with small relative defect and weights that define orders}\label{s:criterion}

Since, for every positive constant $c$, a quasi-character $\tau:G \to \RR$ of defect $D$ can be rescaled ($(c\tau)(g)= c\tau(g)$) to a quasi-character of defect $cD$, it sometimes makes more sense to look at relative measures of ``smallness'' of quasi-characters. For our purposes the following notion (applicable to arbitrary real functions on groups, and not only to quasi-characters) sufficies.

\begin{definition}
Let $G$ be a group. A function $\tau: G \to \RR$ has \emph{small relative defect} if
\begin{enumerate}
\item[\textup{(i)}] $\tau(e)=0$, and
\item[\textup{(ii)}] if $g \neq e$ or $h \neq e$, then
\[
 |\tau(g) + \tau(h) - \tau(gh)| < |\tau(g)| + |\tau(h)|.
\]
\end{enumerate}
\end{definition}

\begin{remark}
We can think of the quantity $\delta(g,h) = |\tau(g)+\tau(h)-\tau(gh)|$ as the defect of $\tau$ at the pair $(g,h)$. Condition (ii) only considers the relative size of the defect of $\tau$ at $(g,h)$ by comparing it to the size of the values of $\tau$ at $g$ and $h$ (hence the name small relative defect).
\end{remark}

\begin{example}
If $\tau:G \to \RR$ is a group homomorphism, then $\tau$ is a function of small relative defect if and only if it is an embedding.

For instance, if $\vect{u} \in \RR^n$ is a vector with coordinates that are linearly independent over $\QQ$, then $\tau_{\vect{u}}: \ZZ^n \to \RR$ given by $\tau_{\vect{u}}(\vect{v}) = \vect{u}\cdot \vect{v}$, where $\cdot$ is the usual dot product, is a function of small relative defect. Of course, every character of $\ZZ^n$ that is a function of small relative defect arises in this way.
\end{example}

Call a function $\tau: G \to \RR$ \emph{proper} if
\[ \tau(g) = 0 \qqiff g=e. \]
Note that every function of small relative defect is proper. Indeed, by setting $h=e$ in condition (ii), we obtain, for every $g \in G$, with $g \neq e$,
\[
 0  < |\tau(g)|.
\]

\begin{example}
No Brooks quasi-character $\phi_w$ of $F_k$, $k \geq 2$, is a function of small relative defect. In fact, no such quasi-character is proper, as there is always a letter $x$ in $\AA$ for which $\phi_w(x)=0$.
\end{example}

\begin{proposition}\label{p:small-delta}
Let $G$ be a group. If $\tau: G \to \RR$ is a quasi-character which is a proper function and whose defect satisfies the inequality
\begin{equation}\label{e:2-inf}
 \delta(\tau) < 2 \inf_{g \neq e}{\tau(g)},
\end{equation}
then $\tau$ is a function of small relative defect.
\end {proposition}

\begin{proof}
Condition (i) and condition (ii), when $g=e$ or $h=e$ (but not both), are satisfied for every proper function, and condition (ii), when $g \neq e$ and $h \neq e$, is satisfied since, in that case,
\[
 \delta(g,h) \leq \delta(\tau) < 2 \inf_{f \neq e}{\tau(f)} \leq \tau(g)+\tau(h). \qedhere
\]
\end{proof}

\begin{corollary}\label{c:tau-is-small}
For each oriented word $u$ of rank $k$, $k \geq 2$, the quasi-character $\tau_u$ is a function of small relative defect.
\end{corollary}

\begin{proof}
By Theorem~\ref{t:tau-is-quasi}, $\delta(\tau_u) = 1/2 < 1 = 2 \inf_{g \neq e}{\tau(g)}$. Since $\tau_u$ is also a proper function, it has small relative defect, by Proposition~\ref{p:small-delta}.
\end{proof}

The reason we are interested in functions with small relative defect is that they can be used as weight functions that define orders on groups.

\begin{proposition}\label{p:criterion}
Let $G$ be a group and $\tau:G \to \RR$ a function with small relative defect. Then $G$ is orderable and an order on $G$ may be defined by declaring
\[
 P = \{\ g \in G \mid \tau(g)>0 \ \}
\]
to be the positive cone of $G$.
\end{proposition}

\begin{proof}
Recall that functions with small relative defect are proper.

By setting $h=g^{-1}$ in condition (ii), we obtain, for every $g \in G$, with $g \neq e$,
\begin{equation}\label{e:gginv-bound}
 |\tau(g)+\tau(g^{-1})| < |\tau(g)|+|\tau(g^{-1})|.
\end{equation}
If $\tau(g)$ and $\tau(g^{-1})$ are both positive or both negative, then $|\tau(g)+\tau(g^{-1})| = |\tau(g)|+|\tau(g^{-1})|$, contradicting inequality~\eqref{e:gginv-bound}. Therefore, $\tau(g)$ and $\tau(g^{-1})$ have opposite signs when $g \neq e$, $P^{-1} = \{ g \in G \mid \tau(g)<0 \}$, and $G$ decomposes into a disjoint union
\[
 G = P^{-1} \sqcup \{e\} \sqcup P.
\]

It remains to show that $P$ is closed under multiplication. Let $g,h \in P$. Then $\tau(g)>0$, $\tau(h)>0$, $g \neq e$, $h \neq e$ and, by condition (ii), $\tau(g) + \tau(h) - \tau(gh) < \tau(g)+\tau(h)$, which implies $\tau(gh) > 0$, hence $gh \in P$.
\end{proof}

\begin{remark}
In a trivial way, every order $\leq $ on every orderable group $G$ comes from a quasi-character of small relative defect. Namely, the function $\tau: G \to \RR$ that evaluates to 1, 0 and -1 on the positive cone, identity, and the negative cone, respectively, is a quasi-character of small relative defect that defines $\leq$ when used as a weight function.
\end{remark}

A direct corollary of Proposition~\ref{p:criterion} and Corolary~\ref{c:tau-is-small} is the following result, already announced in Section~\ref{s:orders}.

\begin{theorem}\label{t:is-order}
Let $u$ be an oriented word of rank $k$, $k \geq2$. The relation $\leq_u$ on $F_k$ defined by
\[
 g \leq_u h \qqiff \tau_u(g^{-1}h) \geq 0
\]
is a left order on $F_k$ with positive cone $P_u=\{g \in G \mid \tau(g)>0\}$.
\end{theorem}

%-----------------------------------------------------------

\section{Lexicographically based orders}\label{s:lex}

We characterize all oriented words of rank $k$ that induce orders extending the lexicographic order on the monoid $\AA^*$ based on $a_1<a_2<\dots<a_k$ (Proposition~\ref{p:lexicographic}). There are exactly $k!$ such oriented words.

We start by observing that the order among the positive (negative) generators agrees with their relative position in the defining word.

\begin{lemma}\label{l:generator-order}
Let $u$ be an oriented word of rank $k$, $k \geq 2$. For $1 \leq i,j\leq k$, $i \neq j$,
\begin{gather*}
 a_i <_u a_j \iff a_i .. a_j \textup{ in } u, \\
 A_i <_u A_j \iff A_i .. A_j \textup{ in } u.
\end{gather*}
\end{lemma}

\begin{proof}
Indeed,
\begin{gather*}
 a_i <_u a_j \iff e <_u A_ia_j \iff \wtc_u(A_ia_j)=0 \iff a_i..a_j \text{ in } u, \\
 A_i <_u A_j \iff e <_u a_iA_j \iff \wtc_u(a_iA_j)=1 \iff A_i..A_j \text{ in } u. \qedhere
\end{gather*}
\end{proof}

\begin{proposition}\label{p:lexicographic}
The order induced by an oriented word $u$ of rank $k$, $k \geq 2$, is an extension of the lexicographic order on the monoid $\AA^*$ based on $a_1<a_2< \dots< a_k$ if and only if the word $u$ has the form
\begin{equation}\label{e:W-lexform}
 u = a_1a_2\dots a_k A_{j_1}A_{j_2}\dots A_{j_k},
\end{equation}
where $j_1,j_2,\dots,j_k$ is some permutation of $1,2,\dots,k$.
\end{proposition}

\begin{proof}($\Longrightarrow$)
Let the order $\leq_u$ on $F_k$ induced by the oriented word $u$ of rank $k$ be an extension of the lexicographic order based on $a_1<a_2< \dots< a_k$.

The relative order of the positive letters in $u$ must be exactly $a_1,a_2,\dots,a_k$, by Lemma~\ref{l:generator-order}.

By way of contradiction, assume $A_i ..a_k$ in $u$, for some $i$, $1\leq i \leq k$. This implies that $\wtc_u(a_ia_k)=1$ and, therefore,
\[
 \tau(A_2a_1a_ia_k) = \wtc_u(A_2a_1) + \wtc_u(a_1a_i) + \wtc_u(a_ia_k) + \frac{1}{2} \geq -1+0+1+\frac{1}{2} = \frac{1}{2}>0,
\]
which implies that
\[
 a_2 <_u a_1a_ia_k,
\]
violating the lexicographic order, a contradiction.

($\Longleftarrow$) Let $u$ be an oriented word of the form~\eqref{e:W-lexform}.

The structure of the word $u$ implies that the only negative contributions to the induced weight $\tau_u$ are given by subwords of the form $A_ja_i$, for $1\leq i < j \leq k$. Therefore, for any words $g$ and $h$ over $\AA$, and $1\leq i < j \leq k$,
\[
 0 < \tau_u(a_ig) \qquad \text{and} \qquad 0 < \tau_u(g^{-1}A_ia_jh),
\]
which implies that, for any words $g$ and $h$ over $\AA$, and $1\leq i < j \leq k$,
\[
 e <_u a_ig <_u a_jh,
\]
and this shows that $\leq_u$ is compatible with the lexicographic order on the monoid $\AA^*$ based on $a_1<a_2< \dots< a_k$.
\end{proof}

\begin{remark}
It is clear that, more generally, the order induced by an oriented word $u$ of rank $k$ is an extension of the lexicographic order on the monoid $\AA^*$ based on $a_{i_1}<a_{i_2}< \dots< a_{i_k}$,  where $i_1,i_2,\dots,i_k$ is some fixed permutation of $1,2,\dots,k$, if and only if the word $u$ has the form
\begin{equation}\label{e:u-lexxform}
 u = a_{i_1}a_{i_2}\dots a_{i_k} A_{j_1}A_{j_2}\dots A_{j_k},
\end{equation}
where $j_1,j_2,\dots,j_k$ is some permutation of $1,2,\dots,k$.

Note that the words of the form~\eqref{e:u-lexxform} are precisely the words $u$ for which the letter transition weight $\beta_u$ vanishes.
\end{remark}

%---------------------------------------------------------------------------

\section{Inverting the defining word and reversible orders}\label{s:rev}

For a group word $g$, denote by $g^{-1}$ its inverse, by $g^\rev$ its reversal and by $g^\tog$ its case switch (i.e., the word obtained when the case of each letter in $g$ is switched). Note that these three operations on words are involutions, they commute, and product of any two is the third one.

We provide a precise relation between the positive cones $P_u$ and $P_{u^{-1}}$ associated to a word $u$ and its inverse $u^{-1}$. Namely these two cones can be obtained from each other by word reversal (Theorem~\ref{t:inverse-reverse}).

A \emph{reversible order} on the free group $F_k$ of rank $k$ is an order on $F_k$ for which the positive cone is closed under word reversal (i.e., $e \leq g \iff e \leq g^\rev$). We characterize all oriented words of rank $k$ that induce reversible orders (Proposition~\ref{p:reversible}). There are exactly $2^k \cdot k!$ such oriented words and they are precisely the words $u$ for which $u=u^{-1}$.

\begin{lemma}\label{l:inverse-reverse}
For every oriented word $u$ of rank $k$, $k \geq 2$, and every reduced word $g$ of rank $k$,

(a) $\alpha_u(g) + \alpha_{u^{-1}}(g^\tog) =0$, \hfill $\beta_u(g) + \beta_{u^{-1}}(g^\tog) =0$, \hfill $\omega(g) + \omega(g^\tog) =0$.

(b) $\tau_u(g) + \tau_{u^{-1}}(g^\tog) =0$.

(c) $\tau_u(g) = \tau_{u^{-1}}(g^\rev)$.
\end{lemma}

\begin{proof}
(a) We have
\begin{align*}
 \alpha_{u^{-1}}(g^\tog) &=
   \sum_{\substack{a,b \in \AA \\ A..B \textup{ in } u^{-1}}} \#_{aB}(g^\tog) - \sum_{\substack{a,b \in \AA \\ a..b \textup{ in } u^{-1}}} \#_{Ba}(g^\tog) \\
 &=\sum_{\substack{a,b \in \AA \\ b..a \textup{ in } u     }} \#_{aB}(g^\tog) - \sum_{\substack{a,b \in \AA \\ B..A \textup{ in } u     }} \#_{Ba}(g^\tog) \\
 &=\sum_{\substack{a,b \in \AA \\ b..a \textup{ in } u     }} \#_{Ab}(g     ) - \sum_{\substack{a,b \in \AA \\ B..A \textup{ in } u     }} \#_{bA}(g)      \\
 &=- \alpha_u(g),
\end{align*}
where the last equality can be obtained by reindexing (exchange the roles of $a$ and $b$).

Similarly,
\begin{align*}
 \beta_{u^{-1}}(g^\tog) &=
   \sum_{\substack{a,b \in \AA \\ A..b \textup{ in } u^{-1}}} \#_{ab}(g^\tog) - \sum_{\substack{a,b \in \AA \\ A..b \textup{ in } u^{-1}}} \#_{BA}(g^\tog) \\
 &=\sum_{\substack{a,b \in \AA \\ B..a \textup{ in } u     }} \#_{ab}(g^\tog) - \sum_{\substack{a,b \in \AA \\ B..a \textup{ in } u     }} \#_{BA}(g^\tog) \\
 &=\sum_{\substack{a,b \in \AA \\ B..a \textup{ in } u     }} \#_{AB}(g     ) - \sum_{\substack{a,b \in \AA \\ B..a \textup{ in } u     }} \#_{ba}(g)      \\
 &=- \beta_u(g).
\end{align*}

Since the last letters in $g$ and $g^\tog$ are $a$ and $A$, or $A$ and $a$, respectively, for some $a \in \AA$, $\omega(g)+\omega(g^\tog)=0$.

(b) Follows directly from (a).

(c) We have $\tau_u(g) = -\tau_{u^{-1}}(g^\tog) = \tau_{u^{-1}}((g^\tog)^{-1}) = \tau_{u^{-1}}(g^\rev)$, where the first equality comes from part (b) and the second from Lemma~\ref{l:gginv}(c).
\end{proof}

\begin{theorem}\label{t:inverse-reverse}
For every oriented word $u$ of rank $k$, $k \geq 2$,
\[
 P_{u^{-1}} = (P_u)^\rev.
\]
\end{theorem}

\begin{proof}
Directly from the last part of the last lemma.
\end{proof}

\begin{proposition}\label{p:reversible}
The positive cone of the order induced by an oriented word $u$ is closed under word reversal if and only if $u = u^{-1}$.
\end{proposition}

\begin{proof}
($\Longrightarrow$) Let $P_u$ be closed under word reversal.

Let $a, b \in \AA$, with $a \neq b$. If both
\[
 a..b \qqand A..B,
\]
then $\tau_u(aB)=1/2$ and $\tau_u(Ba)=-1/2$, implying
\[
 Ba <_u e <_u aB,
\]
which contradicts the assumption that the positive cone is reversible. Therefore,
\begin{equation}\label{e:reversal1}
 a..b \iff B..A.
\end{equation}
Similarly, if $a,b\in\AA$ and both
\[
 a..B \qqand A..b,
\]
then $\tau_u(abAB) = \wtc_u(ab)+\wtc_u(bA)+\wtc_u(AB)-1/2 = 1+\wtc_u(bA)+0-1/2>0$ and $\tau_u(BAba) = \wtc_u(BA)+\wtc_u(Ab)+\wtc_u(ba) +1/2 = -1+\wtc_u(Ab)+0+1/2<0$, implying
\[
 BAba <_u e <_u abAB,
\]
which contradicts the assumption that the positive cone is reversible. Therefore,
\begin{equation}\label{e:reversal2}
 a..B \iff b..A.
\end{equation}
Since~\eqref{e:reversal1} holds for every letter $b$, with $b \neq a$, and~\eqref{e:reversal2} holds for every letter $b$, the set of letters appearing to the right of $a$ is the inverse of the set of letters that appears to the left of $A$, and since this holds for every letter $a$, the word $u$ has the property $u=u^{-1}$.

($\Longleftarrow$) Let $u$ be an oriented word for which $u=u^{-1}$. By Lemma~\ref{l:inverse-reverse}, $(P_u)^\rev= P_{u^{-1}} = P_u$.
\end{proof}

The following is direct corollary of Proposition~\ref{p:lexicographic} and Proposition~\ref{p:reversible}.

\begin{corollary}
The only oriented word of rank $k$ that induces an order on $F_k$ that is an extension of the lexicographic order on the monoid $\AA^*$ based on $a_1<a_2<\dots<a_k$ and for which the positive cone is closed under word reversal is the word
\[
 u = a_1a_2\dots a_k A_k \dots A_2 A_1.
\]
\end{corollary}

Note that, for each $k \geq 2$, $u=a_1a_2\dots a_k A_k \dots A_2A_1$ is the word that induces the order on $F_k$ from~\cite{sunic:free-lex}.

%-------------------------------------------------------------

\section{The orders induced by oriented words are not two-sided}\label{s:no-bi}

\begin{proposition}\label{p:no-bi}
For every oriented word $u$ of rank $k$, $k \geq 2$, the left order $\leq_u$ is not a right order on $F_k$.
\end{proposition}

\begin{proof}
By way of contradiction, assume that $\leq_u$ is a bi-order. Then the positive cone $P_u$ is invariant under conjugation in $F_k$.

Let $a,b\in \AA$, with $a..b$ in $u$. We have
\[
 a..b \textup{ in }u \implies Ab \in P_u  \implies bA \in P_u \implies B..A \textup{ in } u,
\]
where we used the invariance under conjugation in the second implication. Since $Bab \in P_u$ and $\tau_u(Bab) = -1+\wtc_u(ab)+\frac{1}{2}$, we must have $\wtc_u(ab)=1$, which means that $A..b$ in $u$. Since $baB \in P_u$ and $\tau_u(baB) = \wtc_u(ba)+0-\frac{1}{2}$, we must have $\wtc_u(ba)=1$, which means that $B..a$ in $u$. Therefore,
\begin{align*}
 &B..A..b \textup{ in } u, \\
 &B..a..b \textup{ in } u,
\end{align*}
which implies that
\[
 \tau_u(AbaB) = 0+1+0 - \frac{1}{2}>0 \qqand \tau_u(baBA) = 1+0-1 - \frac{1}{2}<0,
\]
contradicting the invariance of $P_u$ under conjugation.
\end{proof}

%-------------------------------------------------------------

\section{Differences between the defined orders}\label{s:different}

We prove that, for every $k$, the $(2k)!$ orders induced by the oriented words on $F_k$ are distinct. The proof is based on finding, for each distinct pair of oriented words $u$ and $u'$ of rank $k$, a group word in $P_u$ that is not in $P_{u'}$, or the other way around (words of length at most four suffice).

\begin{theorem}\label{t:different}
For every $k$, $k \geq2$, and every pair $u$ and $u'$ of distinct oriented words of rank $k$, the orders $\leq_u$ and $\leq_{u'}$, induced by $u$ and $u'$ on $F_k$, are distinct.
\end{theorem}

\begin{proof}
Let $u$ and $u'$ be two distinct oriented words of rank $k$.

By Lemma~\ref{l:generator-order}, if the relative positions of the lowercase letters or the relative positions of the uppercase letters differ in $u$ and $u'$, the induced orders $\leq_u$ and $\leq_{u'}$ are different (the positive cones $P_u$ and $P_{u'}$ already differ on words of length two).

Assume that the relative positions of the lowercase letters and the relative positions of the uppercase letters are the same in $u$ and $u'$.

Assume further, that there exists a letter $a \in \AA$ such that
\begin{alignat*}{2}
 &A..a \qquad &&\textup{ in } u \\
 &a..A        &&\textup{ in } u'.
\end{alignat*}
Let $b \in \AA$, $b \neq a$. If $a..b$ in $u$ and $u'$ then
\[
 \tau_u(Baa) = \frac{1}{2} \qqand \tau_{u'}(Baa) = -\frac{1}{2},
\]
and if $B..A$ in $u$ and $u'$ then
\[
 \tau_u(aaB) = \frac{1}{2} \qqand \tau_{u'}(aaB) = -\frac{1}{2}.
\]
On the other hand, if both $b..a$ and $A..B$ in $u$ and $u'$, then
\begin{alignat*}{2}
    &A..a           &&\textup{ in } u \\
 b..&a..A..B \qquad &&\textup{ in } u',
\end{alignat*}
and this yields
\[
 \tau_u(Abaa) \geq -1+0+1+\frac{1}{2} = \frac{1}{2} \qqand
 \tau_{u'}(Abaa) = -1+0+0+\frac{1}{2} = -\frac{1}{2}.
\]

It remains to consider the case when, for all $a \in \AA$, the relative positions of $a$ and $A$ in $u$ and $u'$ agree, but there is a pair of distinct letters $a,b \in \AA$ such that
\begin{alignat*}{2}
 &A..b \qquad &&\textup{ in } u \\
 &b..A        &&\textup{ in } u'.
\end{alignat*}
If $a..b$ in $u$ and $u'$ then
\[
 \tau_u(Bab) = \frac{1}{2} \qqand \tau_{u'}(Bab) = -\frac{1}{2},
\]
and if $A..B$ in $u$ and $u'$ then
\[
 \tau_u(abA) = \frac{1}{2} \qqand \tau_{u'}(abA) = -\frac{1}{2}.
\]
On the other hand, if both $b..a$ and $B..A$ in $u$ and $u'$, then
\[
 B..A..b..a \qquad \textup{ in } u,
\]
and, since $B..b$ and $A..a$ in $u$, we also must have $B..b$ and $A..a$ in $u'$. Thus
\[
 B..b..A..a \qquad \textup{ in } u',
\]
and this yields
\[
 \tau_u(ABab)    = -1+0+1+\frac{1}{2} = \frac{1}{2} \qqand
 \tau_{u'}(ABab) = -1+0+0+\frac{1}{2} = -\frac{1}{2}.
\]

Therefore, in each case, there is a reduced word $g$ of rank $k$ for which the signs of the weights $\tau_u(g)$ and $\tau_{u'}(g)$ differ, which implies that the positive cones $P_u$ and $P_{u'}$ differ as well.
\end{proof}

\begin{remark}
Note that it is much easier to show that the orders associated to $u$ and $u'$ are different when the relative positions of the positive letters or the negative letters in $u$ and $u'$ are different. In other words, the slightly more complicated cases in the proof above occur when the case transition weights $\alpha_u$ and $\alpha_{u'}$ coincide and one has to rely only on the letter transition weights $\beta_u$ and $\beta_{u'}$ to make a distinction between $\leq_u$ and $\leq_{u'}$.
\end{remark}

\begin{remark}
Since the letter transition weights are sums of Brooks quasi-characters, it is clear that the difference $\beta_u-\beta_{u'}$ is always a sum of Brooks quasi-characters. It turns out that the difference of the case transition weights $\alpha_u - \alpha_{u'}$ is always a sum of Brooks quasi-characters as well.
\end{remark}

\begin{proposition}
For any two oriented words $u$ and $u'$ of rank $k$, $k \geq 2$, the difference between the weight function $\tau_u$ and $\tau_{u'}$ is a sum of Brooks quasi-characters. More precisely,
\[
 \tau_u - \tau_{u'} = \sum_{\substack{x..y \textup{ in }u \\ y..x \textup{ in }u'}} \phi_{x^{-1}y}
\]
\end{proposition}

\begin{proof}
We have
\[
 \sum_{\substack{a,b \in \AA \\ A..B \textup{ in } u}} \#_{aB}(g) - \sum_{\substack{a,b \in \AA \\ A..B \textup{ in } u'}} \#_{aB}(g) =
 \sum_{\substack{A..B \textup{ in } u\\ B..A \textup{ in } u'}} (\#_{aB}(g) - \#_{bA}(g)) = \sum_{\substack{A..B \textup{ in } u \\ B..A \textup{ in } u'}} \phi_{aB}(g),
\]
since the terms of the form $\#_{aB}(g)$ in the sums corresponding to $u$ and $u'$ cancel when the relative positions of $A$ and $B$ in $u$ and $u'$ agree. Similarly,
\[
 -\sum_{\substack{a,b \in \AA \\ a..b \textup{ in } u}} \#_{Ba}(g) + \sum_{\substack{a,b \in \AA \\ a..b \textup{ in } u'}} \#_{Ba}(g) =
 \sum_{\substack{a..b \textup{ in } u\\ b..a \textup{ in } u'}} (\#_{Ab}(g)-\#_{Ba}(g)) = \sum_{\substack{a..b \textup{ in } u \\ b..a \textup{ in } u'}} \phi_{Ab}(g).
\]
Therefore,
\[
 \alpha_u - \alpha_{u'} = \sum_{\substack{A..B \textup{ in } u \\ B..A \textup{ in } u'}} \phi_{aB} + \sum_{\substack{a..b \textup{ in } u \\ b..a \textup{ in } u'}} \phi_{Ab}.
\]

We also have
\[
 \beta_u - \beta_{u'} = \sum_{\substack{a,b \in \AA \\ A..b \textup{ in } u}} \phi_{ab} - \sum_{\substack{a,b \in \AA \\ A..b \textup{ in } u}} \phi_{ab} =
  \sum_{\substack{A..b \textup{ in } u \\ b..A \textup{ in } u'}} \phi_{ab} - \sum_{\substack{b..A \textup{ in } u \\ A..b \textup{ in } u'}} \phi_{ab},
\]
and, since $-\phi_{ab} = \phi_{BA}$,
\[
 \beta_u - \beta_{u'} =
   \sum_{\substack{A..b \textup{ in } u \\ b..A \textup{ in } u'}} \phi_{ab} + \sum_{\substack{b..A \textup{ in } u \\ A..b \textup{ in } u'}} \phi_{BA}. \qedhere
\]
\end{proof}

We provide relations between $\tau_u$ and $\tau_{u^\rev}$, $\tau_{u^\tog}$ and $\tau_{u^{-1}}$. These relations provide, at least indirectly, relations between the positive cones associated to $u$, $u^\rev$, $u^\tog$, and $u^{-1}$.

Note that the sum of Brooks quasi-morphisms $\sum_{a \in \AA} \phi_a$ is the \emph{total exponent} homomorphism $F_k \to \RR$ calculating the difference between the the number of positive and negative letters in its argument.

\begin{proposition}
Let $u$ be an oriented word of rank $k$, $k \geq 2$. Then

(a)
\[
 \tau_u + \tau_{u^\rev} = \sum_{a \in \AA} \phi_a,
\]

(b)
\[
 \tau_u(g) + \tau_{u^{-1}}(g^\tog) = 0,
\]

(c)
\[
 \tau_u(g) + \tau_{u^\tog}(g^\rev) = \sum_{a \in \AA} \phi_a(g).
\]
\end{proposition}

\begin{proof}
(a) For every reduced word $g$ of rank $k$,
\[
 \alpha_{u^\rev}(g) =
  \sum_{\substack{a,b \in \AA \\A..B \textup{ in } u^\rev}} \#{aB}(g) - \sum_{\substack{a,b \in \AA \\a..b \textup{ in } u^\rev}} \#{Ba}(g) =
  \sum_{\substack{a,b \in \AA \\B..A \textup{ in } u  }} \#{aB}(g) - \sum_{\substack{a,b \in \AA \\b..a \textup{ in } u     }} \#{Ba}(g).
\]
Since, for every pair of distinct letters $a,b \in \AA$, either $A..B$ or $B..A$ in $u$, but not both, and either $a..b$ or $b..a$ in $u$, but not both,
\[
 \alpha_u(g) + \alpha_{u^\rev}(g) = \sum_{\substack{a,b \in \AA \\ a \neq b}} \#_{aB}(g) - \sum_{\substack{a,b \in \AA \\ a \neq b}} \#_{Ba}(g).
\]

Similarly,
\[
 \beta_{u^\rev}(g) =
  \sum_{\substack{a,b \in \AA \\A..b \textup{ in } u^\rev}} \#{ab}(g) - \sum_{\substack{a,b \in \AA \\A..b \textup{ in } u^\rev}} \#{BA}(g) =
  \sum_{\substack{a,b \in \AA \\b..A \textup{ in } u  }} \#{aB}(g) - \sum_{\substack{a,b \in \AA \\b..A \textup{ in } u  }} \#{BA}(g)
\]
and
\[
 \beta_u(g) + \beta_{u^\rev}(g) = \sum_{a,b \in \AA} \#_{ab}(g) - \sum_{a,b \in \AA} \#_{BA}(g).
\]

Therefore
\[
 \alpha_u(g) + \beta_u(g) + \alpha_{u^\rev}(g) + \beta_{u^\rev}(g) =
 \sum_{\substack{a \in \AA,~x \in \AA^\pm \\ a \neq x^{-1}} } \#{ax}(g) - \sum_{\substack{b \in \AA,~y \in \AA^\pm \\ b \neq y^{-1}}} \#{By}(g).
\]
The first sum above counts each positive letter in $g$, except the last one, with plus sign, and the second one counts each negative letter, except the last one, with a minus sign. Since $2 \omega$ counts the last letter with appropriate sign, the conclusion follows.

(b) This is established in Lemma~\ref{l:inverse-reverse}(b).

(c) By part (b),
\[
 \tau_{u^\tog}(g^\rev) = - \tau_{(u^\tog)^{-1}}((g^\rev)^\tog) = - \tau_{u^\rev}(g^{-1}) = \tau_{u^\rev}(g),
\]
and the conclusion follows from part (a).
\end{proof}

%---------------------------------------------------------------------------

\def\cprime{$'$}

%\bibliographystyle{alpha}
%\bibliography{../smath}

\end{document}